\documentclass[10pt,oneside]{amsart}

\usepackage[
paper=a4paper,
headsep=15pt,text={135mm,216mm},centering
]{geometry}

\def\VR{\kern-\arraycolsep\strut\vrule &\kern-\arraycolsep}
\def\vr{\kern-\arraycolsep & \kern-\arraycolsep}

\usepackage{graphicx, pinlabel, color, xcolor}
\usepackage[bookmarks]{hyperref}
\usepackage{amssymb,amsxtra}
\usepackage{tikz}
\usetikzlibrary{
  cd,
  calc,
  positioning,
  arrows,
  decorations.pathreplacing,
  decorations.markings,
}
\usepackage[mode=buildnew]{standalone}

\definecolor{todo-background-color}{gray}{0.95}
\usepackage[
  textwidth=1.1in,
  backgroundcolor=todo-background-color,
  bordercolor=black,
  linecolor=black,
  textsize=footnotesize
]{todonotes}

\iftrue
    
    \makeatletter
    \def\@settitle{%
      \vspace*{-10pt}
      \begin{flushleft}%
        \LARGE\bfseries
        \strut\@title\strut
      \end{flushleft}%
    }
    \def\@setauthors{%
      \begingroup
      \def\thanks{\protect\thanks@warning}%
      \trivlist
      \raggedright
      \large \@topsep27\p@\relax
      \advance\@topsep by -\baselineskip
    \item\relax
      \author@andify\authors
      \def\\{\protect\linebreak}%
      \authors
      \ifx\@empty\contribs
      \else
      ,\penalty-3 \space \@setcontribs
      \@closetoccontribs
      \fi
      \normalfont
      \endtrivlist
      \endgroup
    }
    \def\@setaddresses{\par
      \nobreak \begingroup
      \small\raggedright
      \def\author##1{\nobreak\addvspace\smallskipamount}%
      \def\\{\unskip, \ignorespaces}%
      \interlinepenalty\@M
      \def\address##1##2{\begingroup
        \par\addvspace\bigskipamount\noindent
        \@ifnotempty{##1}{(\ignorespaces##1\unskip) }%
        {\ignorespaces##2}\par\endgroup}%
      \def\curraddr##1##2{\begingroup
        \@ifnotempty{##2}{\nobreak\noindent\curraddrname
          \@ifnotempty{##1}{, \ignorespaces##1\unskip}\/:\space
          ##2\par}\endgroup}%
      \def\email##1##2{\begingroup
        \@ifnotempty{##2}{\nobreak\noindent E-mail address%
          \@ifnotempty{##1}{, \ignorespaces##1\unskip}\/:\space
          \ttfamily##2\par}\endgroup}%
      \def\urladdr##1##2{\begingroup
        \def~{\char`\~}%
        \@ifnotempty{##2}{\nobreak\noindent\urladdrname
          \@ifnotempty{##1}{, \ignorespaces##1\unskip}\/:\space
          \ttfamily##2\par}\endgroup}%
      \addresses
      \endgroup
      \global\let\addresses=\@empty
    }
    \def\@setabstracta{%
      \ifvoid\abstractbox
      \else
      \skip@17pt \advance\skip@-\lastskip
      \advance\skip@-\baselineskip \vskip\skip@
      \box\abstractbox
      \prevdepth\z@ 
      \vskip-28pt
      \fi
    }
    \renewenvironment{abstract}{%
      \ifx\maketitle\relax
      \ClassWarning{\@classname}{Abstract should precede
        \protect\maketitle\space in AMS document classes; reported}%
      \fi
      \global\setbox\abstractbox=\vtop \bgroup
      \normalfont\small
      \list{}{\labelwidth\z@
        \leftmargin0pc \rightmargin\leftmargin
        \listparindent\normalparindent \itemindent\z@
        \parsep\z@ \@plus\p@
        
      }%
    \item[\hskip\labelsep\bfseries\abstractname.]%
    }{%
      \endlist\egroup
      \ifx\@setabstract\relax \@setabstracta \fi
    }

    \def\ps@headings{\ps@empty
      \def\@evenhead{%
        \setTrue{runhead}%
        \normalfont\scriptsize
        \rlap{\thepage}\hfill
        \def\thanks{\protect\thanks@warning}%
        \leftmark{}{}}%
      \def\@oddhead{%
        \setTrue{runhead}%
        \normalfont\scriptsize
        \def\thanks{\protect\thanks@warning}%
        \rightmark{}{}\hfill \llap{\thepage}}%
      \let\@mkboth\markboth
    }\ps@headings

    \def\section{\@startsection{section}{1}%
      \z@{-1.4\linespacing\@plus-.5\linespacing}{.8\linespacing}%
      {\normalfont\bfseries\Large}}
    \def\subsection{\@startsection{subsection}{2}%
      \z@{-.8\linespacing\@plus-.3\linespacing}{.5\linespacing\@plus.2\linespacing}%
      {\normalfont\bfseries\large}}
    \def\subsubsection{\@startsection{subsubsection}{3}%
      \z@{.7\linespacing\@plus.2\linespacing}{-1.5ex}%
      {\normalfont\itshape}}
    \def\paragraph{\@startsection{paragraph}{4}%
      \z@{.7\linespacing\@plus.2\linespacing}{-1.5ex}%
      {\normalfont\itshape}}
    \def\@secnumfont{\bfseries}

    \renewcommand\contentsnamefont{\bfseries}
    \def\@starttoc#1#2{\begingroup
      \setTrue{#1}%
      \par\removelastskip\vskip\z@skip
      \@startsection{}\@M\z@{\linespacing\@plus\linespacing}%
      {.5\linespacing}{
        \contentsnamefont}{#2}%
      \ifx\contentsname#2%
      \else \addcontentsline{toc}{section}{#2}\fi
      \makeatletter
      \@input{\jobname.#1}%
      \if@filesw
      \@xp\newwrite\csname tf@#1\endcsname
      \immediate\@xp\openout\csname tf@#1\endcsname \jobname.#1\relax
      \fi
      \global\@nobreakfalse \endgroup
      \addvspace{32\p@\@plus14\p@}%
      \let\tableofcontents\relax
    }
    \def\contentsname{Contents}
    \def\l@section{\@tocline{2}{.5ex}{0mm}{5pc}{}}
    \def\l@subsection{\@tocline{2}{0pt}{2em}{5pc}{}}
    \makeatother

\fi

\makeatletter
\newcommand{\shortxra}[2][]{\ext@arrow 0359\rightarrowfill@{#1}{#2}}
\def\longrightarrowfill@{\arrowfill@\relbar\relbar\longrightarrow}
\newcommand{\longxra}[2][]{\ext@arrow 0359\longrightarrowfill@{#1}{#2}}

\makeatother

\makeatletter
\def\addtagsub#1{\let\oldtf=\tagform@\def\tagform@##1{\oldtf{##1}\hbox{$_{#1}$}}}
\makeatother

\makeatletter
\def\Nopagebreak{\@nobreaktrue\nopagebreak}
\makeatother

\makeatletter
\newtheoremstyle{theorem-giventitle}
        {}{}              
        {\itshape}                      
        {}                              
        {\bfseries}                     
        {.}                             
        {\thm@headsep}                             
        {\thmnote{\bfseries#3}}
\newtheoremstyle{theorem-givenlabel}
        {}{}              
        {\itshape}                      
        {}                              
        {\bfseries}                     
        {.}                             
        {\thm@headsep}                             
        {\thmname{#1}~\thmnumber{#3}\setcurrentlabel{#3}}
\newtheoremstyle{definition-giventitle}
        {}{}              
        {}                      
        {}                              
        {\bfseries}                     
        {.}                             
        {\thm@headsep}                             
        {\thmnote{\bfseries#3}}
\def\setcurrentlabel#1{\gdef\@currentlabel{#1}}
\makeatother

\newtheorem{theorem}{Theorem}[section]

\newtheorem{proposition}[theorem]{Proposition}

\theoremstyle{definition}

\newtheorem*{case2'}{Case 2$'$}

\theoremstyle{theorem-giventitle}
\newtheorem{theorem-named}{}
\theoremstyle{theorem-givenlabel}
\newtheorem{theorem-labeled}{Theorem}

\theoremstyle{definition-giventitle}
\newtheorem{definition-named}{}
\newtheorem{conjecture-named}{}
\newtheorem{case-named}{}

\numberwithin{equation}{section}

\newcommand{\vol}{\operatorname{Vol}}

\def\sm{\smallsetminus}

\makeatletter

\begin{document}

\title{Homology spheres and property R}

\author{Min Hoon Kim}
\address{School of Mathematics, Korea Institute for Advanced Study, Seoul, Republic of Korea}
\email{kminhoon@kias.re.kr}
\author{JungHwan Park}
\address{School of Mathematics, Georgia Institute of Technology, Atlanta, Georgia, United States}
\email{junghwan.park@math.gatech.edu}
\def\subjclassname{\textup{2010} Mathematics Subject Classification}
\expandafter\let\csname subjclassname@1991\endcsname=\subjclassname
\expandafter\let\csname subjclassname@2000\endcsname=\subjclassname

\begin{abstract} We present infinitely many homology spheres which contain two distinct knots whose $0$-surgeries are $S^1 \times S^2$. This resolves a question posed by Kirby and Melvin in $1978$.
\end{abstract}

\maketitle

\section{Introduction}\label{sec:intro}
A knot $K$ in $S^3$ is said to satisfy \emph{property R}, if surgery on $K$ cannot give $S^1\times S^2$. A celebrated result of Gabai \cite{Gabai:1987-1} states that every non-trivial knot in $S^3$ satisfies property R. Now, we replace $S^3$ with a homology sphere. A homology sphere $M$ contains a knot whose $0$-surgery is $S^1\times S^2$ if and only if $M$ is the boundary of a contractible $4$-manifold with a $0$, $1$, and $2$-handle. The natural question for such $M$, asked by Kirby and Melvin in \cite{Kirby-Melvin:1978-1} (see also \cite[Problem~1.16]{Kirby:1997-1}), is if there is only one knot in $M$ that can produce $S^1\times S^2$ up to equivalence. In this article, we answer this question in the negative. Recall that two knots in a homology sphere $M$ are \emph{equivalent} if there is an orientation-preserving homeomorphism of $M$ that takes one knot to the other. 

\begin{theorem}\label{thm:main} There exist infinitely many homology spheres which contain two distinct knots whose $0$-surgeries are $S^1 \times S^2$.
\end{theorem}

The proof of the theorem proceeds in two steps. First, we construct two component links with unknotted components and linking number $1$. The homology spheres obtained by performing the $0$-surgeries on the links have a property that the $0$-surgery of the meridian of either component of the link is $S^1 \times S^2$. The second step is to show that the meridians are inequivalent knots in infinitely many such homology spheres. This is achieved by using Thurston's hyperbolic Dehn surgery theorem \cite{Thurston:1978-1} and the uniqueness of the JSJ decomposition of a 3-manifold  \cite{Jaco-Shalen:1978-1, Johannson:1979-1}.

\subsubsection*{Acknowledgments}
This project started when the first named author was visiting the Georgia Institute of Technology and he thanks the Georgia Institute of Technology for its generous hospitality and support. We would also like to thank Stefan Friedl, Kouki Sato, and Jennifer Hom for helpful conversations. Kyle Hayden, Thomas E. Mark, and Lisa Piccirillo informed us that they have an independent proof of Theorem~\ref{thm:main} in their paper on exotic Mazur manifolds \cite{Hayden-Mark-Piccirillo:2019-1} which appeared on the ar{X}iv after the first version of this article.

\section{Proof of Theorem~\ref{thm:main}}\label{sec:proof}

\begin{figure}[htb]
\includegraphics{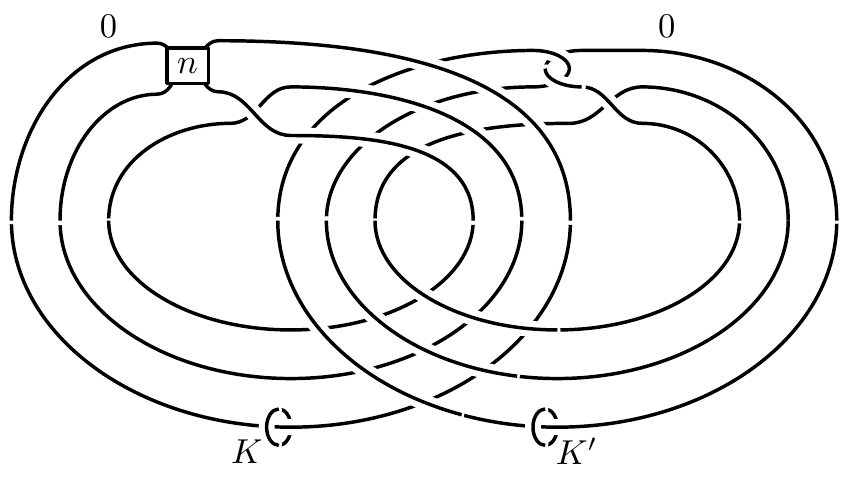}
\caption{Two knots $K$ and $K'$ in a homology sphere $M_n$.}\label{fig:0}
\end{figure}


For each integer $n$, we consider the homology sphere $M_n$ and the knots $K$ and $K'$ in $M_n$ described in Figure~\ref{fig:0}. We show in Proposition~\ref{prop:S1S2} that the $0$-surgeries on $M_n$ along $K$ and $K'$ are $S^1 \times S^2$. We complete the proof by observing that there is a strictly increasing sequence $\{n_i\}_{i=1}^\infty$ of integers such that the homology spheres $M_{n_i}$ are mutually distinct and the knots $K$ and $K'$ are not equivalent.

\begin{proposition}\label{prop:S1S2} Let $K$ and $K'$ be the knots in $M_n$ described in Figure~\ref{fig:0}, then the $0$-surgeries on $K$ and $K'$ are $S^1 \times S^2$.
\end{proposition}
\begin{proof} Consider the surgery diagram of $M_n$ given in Figure~\ref{fig:0}. Each component of the surgery link is unknotted and $K$ and $K'$ are the meridians of the components. Since $K$ is the meridian of the left component, doing $0$-surgery on $K$ removes the left component from the diagram. The result is $S^1\times S^2$ as the right component is unknotted. The argument is symmetric for $K'$.
\end{proof}


It remains to find a sequence $\{n_i\}_{i=1}^\infty$ so that the knots $K$ and $K'$ are distinct and $M_{n_i}$ are mutually distinct. For this purpose, we use Thurston's hyperbolic Dehn filling theorem \cite{Thurston:1978-1} which we recall for completeness.

Let $M$ be a cusped hyperbolic $3$-manifold with $\ell$ cusps. Let $\partial \overline{M} = T_1 \cup T_2 \cup \cdots \cup T_\ell$ and fix generators $\mu_i, \lambda_i$ of fundamental group $\pi_1(T_i)$, for each $i$. Let $(s_1, s_2, \ldots, s_{\ell})$ be a sequence where $s_i$ is either the symbol $\infty$ or a rational number $\frac{p_i}{q_i}$ such that $p_i$ and $q_i$ are coprime integers. Let $M(s_1, s_2, \ldots, s_{\ell})$ be the 3-manifold obtained from $M$ by filling in $T_i$ with a solid torus using the slope $\frac{p_i}{q_i}$ if $s_i=\frac{p_i}{q_i}$ and by not filling in if $s_i=\infty$. 

\begin{theorem}[{\cite{Thurston:1978-1}}]\label{thm:thurston} Let $M$ be a cusped hyperbolic $3$-manifold with $\ell$ cusps. Then $M(s_1, s_2, \ldots, s_{\ell})$ is hyperbolic for all but finitely many sequences $(s_1,s_2,\ldots,s_{\ell})$. Moreover, 
$$\vol(M(s_1, s_2, \ldots, s_{\ell}))\nearrow\vol(M)$$ as all $p^2_i+q^2_i$ approach infinity for all nontrivial framings $\frac{p_i}{q_i}$.
\end{theorem}

\begin{figure}[htb]
\includegraphics{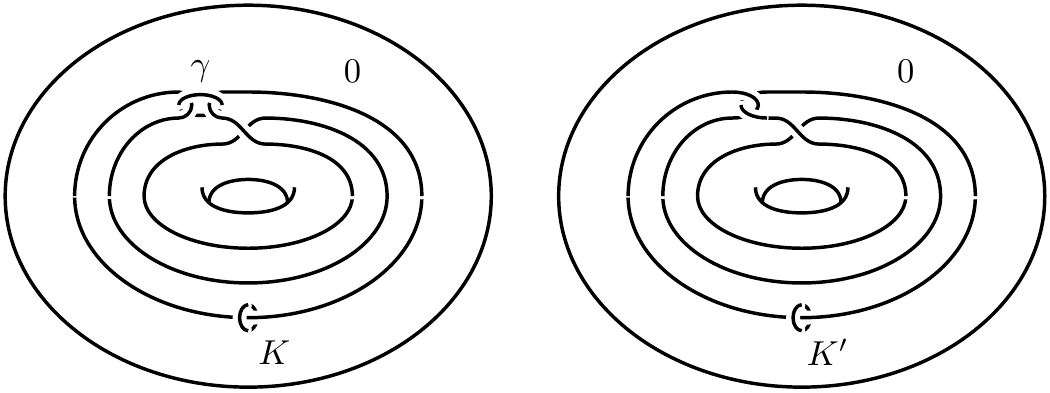}
\caption{The 3-manifolds $N$ (left) and $N'$ (right).}\label{fig:JSJ}
\end{figure} 
To decompose $M_n\sm K$ and $M_n\sm K'$ along an embedded torus, we consider the 3-manifolds $N$ and $N'$ with tori boundary given by the surgery diagrams of Figure~\ref{fig:JSJ}. In Figure~\ref{fig:JSJ}, the surgery links are in $S^1\times D^2$ so that they represent 3-manifolds with tori boundary.

\begin{figure}[htb]
\includegraphics{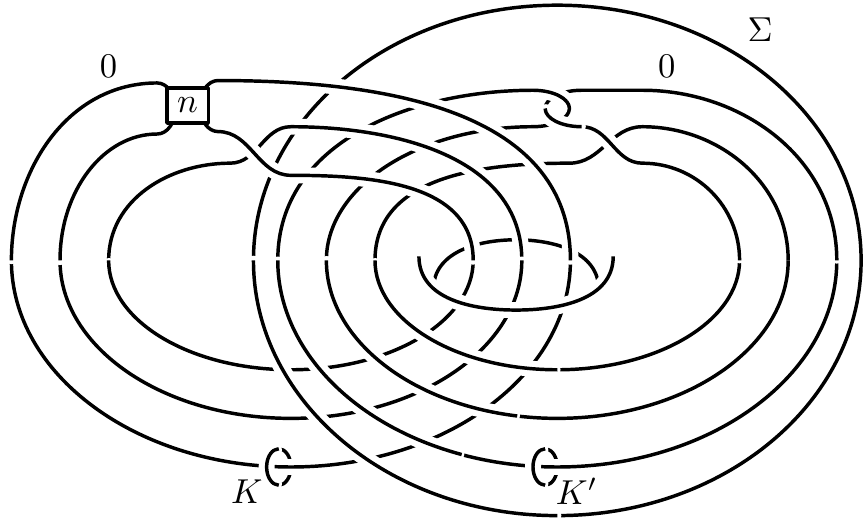}
\caption{An embedded torus $\Sigma$ in $M_n$.}\label{fig:Mn}
\end{figure}

Let $\Sigma\subset M_n$ be an embedded torus in Figure~\ref{fig:Mn}. For each $n\geq 1$, $M_n$ decomposes into
\[M_n=N_n\cup_{\Sigma} N'\]
 where $N_n$ is the 3-manifold obtained from the left diagram of Figure~\ref{fig:JSJ} by performing $-\frac{1}{n}$-surgery on $\gamma$. As depicted in Figure~\ref{fig:JSJ}, we remark that the knots $K$ and $K'$ lie in $N_n$ and $N'$, respectively.
\begin{proposition}\label{prop:JSJ}
For sufficiently large $n$, the knot exteriors $M_n\smallsetminus K$ and $M_n\smallsetminus K'$ have JSJ decompositions $(N_n\smallsetminus K )\cup_\Sigma N'$ and $N_n \cup_\Sigma (N'\smallsetminus K')$, respectively. Moreover, each of the JSJ pieces and $N\smallsetminus \gamma$ are hyperbolic and $\vol(N\smallsetminus \gamma)=7.32772\cdots$ and $\vol(N')=3.66386\cdots$.
\end{proposition}

\begin{proof} By using Snappy \cite{SnapPy} within Sage, we have checked that $N\smallsetminus \gamma$, $N\smallsetminus (\gamma\cup K)$, $N'$, and $N'\smallsetminus K'$ are hyperbolic and $\vol(N\smallsetminus \gamma)=7.32772\cdots$ and $\vol(N')=3.66386\cdots$:
\begin{verbatim}
sage: Ng=snappy.Manifold(`N-gamma.tri')
sage: Ng.verify_hyperbolicity()
(True,
 [1.1126478571421? + 0.5253107612663?*I,
  0.3377160281882? + 0.5374734379681?*I,
  0.9103587832338? + 0.7387973827458?*I,
  0.39026985863857? + 0.18005453683689?*I,
  0.9581510078158? + 0.5060866938299?*I,
  0.4067155957530? + 0.1767530520671?*I,
  -1.1126478571421? + 0.9746892387338?*I,
  0.1839808394078? + 0.4310139380090?*I,
  0.5481237393974? + 0.46122162314913?*I,
  -0.3902698586386? + 1.8199454631631?*I])
sage: Ng.volume()
7.32772475341775

sage: NgK=snappy.Manifold(`N-gammaUK.tri')
sage: NgK.verify_hyperbolicity()
(True,
 [0.6285329320609? + 1.2645427568179?*I,
  0.3930756888787? + 1.1360098247571?*I,
  0.3658649082287? + 0.6848071871001?*I,
  0.2138486222426? + 0.7279803504860?*I,
  0.7279803504860? + 0.7861513777575?*I,
  0.3658649082287? + 0.6848071871001?*I,
  0.6285329320609? + 1.2645427568179?*I,
  0.3658649082287? + 0.6848071871001?*I,
  0.3930756888787? + 1.1360098247571?*I,
  0.2138486222426? + 0.7279803504860?*I])

sage: Nprime=snappy.Manifold(`Nprime.tri')
sage: Nprime.verify_hyperbolicity()
(True,
 [0.?e-14 + 2.00000000000000?*I,
  -1.00000000000000? + 2.00000000000000?*I,
  1.20000000000000? + 0.40000000000000?*I,
  1.00000000000000? + 0.50000000000000?*I,
  0.20000000000000? + 0.40000000000000?*I])
sage: Nprime.volume()
3.66386237670887

sage: NprimeKprime=snappy.Manifold(`Nprime-Kprime.tri')
sage: NprimeKprime.verify_hyperbolicity()
(True,
 [0.23278561593839? + 0.79255199251545?*I,
  0.65883609808599? + 1.16154139999725?*I,
  0.63054057134146? + 0.65136446417090?*I,
  0.65883609808599? + 1.16154139999725?*I,
  0.02829552674454? + 0.51017693582636?*I])\end{verbatim}

By Theorem~\ref{thm:thurston}, $N_n\smallsetminus K$ and $N_n$ are also hyperbolic for sufficiently large $n$. The proof is complete by noting that gluing two hyperbolic pieces gives a JSJ decomposition.\end{proof}

We are now ready to prove Theorem~\ref{thm:main}. 

\begin{proof}[Proof of Theorem~\ref{thm:main}] 
Let $K$ and $K'$ be the knots in the homology sphere $M_n$ as described above.
By Proposition~\ref{prop:S1S2}, $0$-surgeries on $K$ and $K'$ are $S^1\times S^2$. Suppose they are equivalent, then by the uniqueness of JSJ decomposition and Proposition~\ref{prop:JSJ}, $N_n$ is homeomorphic to $N'$, for large enough $n$. This is not possible since by Proposition~\ref{prop:JSJ} \[\vol(N_n)\nearrow\vol(N\smallsetminus \gamma)=7.32772\cdots \text{ and }\vol(N')=3.66386\cdots.\]

Lastly, we show that there exists a sequence  $\{n_i\}_{i=1}^\infty$ where $M_{n_i}$ and $M_{n_j}$ are homeomorphic if and only if $n_i = n_j$. Since $N\smallsetminus \gamma$ is hyperbolic by Proposition~\ref{prop:JSJ}, for large enough $n$, $N_n$ is hyperbolic by Theorem~\ref{thm:thurston} and $M_n$ has a JSJ decomposition $N_n \cup N'$. Moreover, by the uniqueness of JSJ decomposition and by Theorem~\ref{thm:thurston}, there exists a sequence $\{n_i\}_{i=1}^\infty$ where $M_{n_i}$ and $M_{n_j}$ are homeomorphic if and only if $n_i = n_j$. The proof is completed, by choosing the homology spheres to be $M_{n_i}$ and the knots to be $K$ and $K'$.\end{proof}

\bibliographystyle{amsalpha}
\def\MR#1{}
\bibliography{bib}

\providecommand{\bysame}{\leavevmode\hbox to3em{\hrulefill}\thinspace}
\providecommand{\MR}{\relax\ifhmode\unskip\space\fi MR }
\providecommand{\MRhref}[2]{%
  \href{http://www.ams.org/mathscinet-getitem?mr=#1}{#2}
}
\providecommand{\href}[2]{#2}
\begin{thebibliography}{HMP19}

\bibitem[CDW]{SnapPy}
Marc Culler, Nathan~M. Dunfield, and Jeffrey~R. Weeks, \emph{Snappy, a computer
  program for studying the topology of $3$-manifolds.}

\bibitem[Gab87]{Gabai:1987-1}
David Gabai, \emph{Foliations and the topology of {$3$}-manifolds.
  \textup{III}}, J. Differential Geom. \textbf{26} (1987), no.~3, 479--536.
  \MR{910018}

\bibitem[HMP19]{Hayden-Mark-Piccirillo:2019-1}
Kyle Hayden, Thomas~E. Mark, and Lisa Piccirillo, \emph{Exotic {M}azur
  manifolds and knot trace invariants}, ar{X}iv:1908.05269, 2019.

\bibitem[Joh79]{Johannson:1979-1}
Klaus Johannson, \emph{Homotopy equivalences of {$3$}-manifolds with
  boundaries}, Lecture Notes in Mathematics, vol. 761, Springer, Berlin, 1979.
  \MR{551744}

\bibitem[JS78]{Jaco-Shalen:1978-1}
William Jaco and Peter~B. Shalen, \emph{A new decomposition theorem for
  irreducible sufficiently-large {$3$}-manifolds}, Algebraic and geometric
  topology ({P}roc. {S}ympos. {P}ure {M}ath., {S}tanford {U}niv., {S}tanford,
  {C}alif., 1976), {P}art 2, Proc. Sympos. Pure Math., XXXII, Amer. Math. Soc.,
  Providence, R.I., 1978, pp.~71--84. \MR{520524}

\bibitem[Kir97]{Kirby:1997-1}
Robion Kirby, \emph{Problems in low-dimensional topology}, Geometric topology
  ({A}thens, {GA}, 1993), AMS/IP Stud. Adv. Math., vol.~2, Amer. Math. Soc.,
  Providence, RI, 1997, pp.~35--473. \MR{1470751}

\bibitem[KM78]{Kirby-Melvin:1978-1}
Robion Kirby and Paul Melvin, \emph{Slice knots and property {${\rm R}$}},
  Invent. Math. \textbf{45} (1978), no.~1, 57--59. \MR{0467754}

\bibitem[Thu78]{Thurston:1978-1}
William Thurston, \emph{The geometry and topology of three--manifolds},
  Princeton Univ. Math. Dept., 1978.

\end{thebibliography}

\end{document}